\documentclass[oneside,reqno]{amsart}
\usepackage{amsmath, amsthm, mathrsfs, amssymb, amsfonts, mathtools}
\usepackage{graphicx}
\graphicspath{ {Figures/} }
\usepackage{hyperref}
\hypersetup{colorlinks=true, linkcolor=blue, citecolor=magenta, filecolor=magenta, urlcolor=cyan}

\usepackage[backend=bibtex, style=phys, doi=false, isbn=false, hyperref=true, backref=false, giveninits=true, abbreviate=true, sorting=nyt]{biblatex}
\def\transpose{\intercal}
\def\diff{\mathrm{d}}
\DeclareMathOperator{\cn}{cn}
\DeclareMathOperator{\sn}{sn}

\newtheorem{lemma}{Lemma}

\usepackage{orcidlink}

\usepackage{afterpage}

\numberwithin{equation}{section}
\mathtoolsset{showonlyrefs}

\begin{document}

\title[Periodic Solutions for the 1d Cubic Wave Equation with Dirichlet BC]{Periodic Solutions for the 1d Cubic Wave Equation with Dirichlet Boundary Conditions}

\author{Filip Ficek\orcidlink{0000-0001-5885-7064}}
\address{University of Vienna, Faculty of Mathematics, Oskar-Morgenstern-Platz 1, 1090 Vienna, Austria}
\address{University of Vienna, Gravitational Physics, Boltzmanngasse 5, 1090 Vienna, Austria}
\email[]{filip.ficek@univie.ac.at}
\author{Maciej Maliborski$^{\dagger}$\orcidlink{0000-0002-8621-9761}}
\address{University of Vienna, Faculty of Mathematics, Oskar-Morgenstern-Platz 1, 1090 Vienna, Austria}
\address{University of Vienna, Gravitational Physics, Boltzmanngasse 5, 1090 Vienna, Austria}
\email[]{maciej.maliborski@univie.ac.at}
\thanks{We acknowledge the support of the Austrian Science Fund (FWF) through Project \href{http://doi.org/10.55776/P36455}{P 36455} and the START-Project \href{http://doi.org/10.55776/Y963}{Y 963}.}
\thanks{$^{\dagger}$Corresponding author.}


\begin{abstract}
We study time-periodic solutions for the cubic wave equation on an interval with Dirichlet boundary conditions. We begin by following the perturbative construction of Vernov and Khrustalev and provide a rigorous derivation of the fourth-order expansion in small amplitude, which we use to verify the Galerkin scheme. In the main part, we focus on exploring large solutions numerically. We find an intricate bifurcation structure of time-periodic solutions forming a fractal-like pattern and explore it for the first time. Our results suggest that time-periodic solutions exist for arbitrary frequencies, with appearance of fine bifurcation structure likely related to the Cantor set families of solutions described in previous rigorous works.
\end{abstract}

\date{\today}

\maketitle

\tableofcontents

\section{Introduction}
\label{sec:Introduction}

We consider a 1d cubic wave equation on an interval
\begin{equation}
	\label{eq:23.02.14_01}
	\partial_{t}^{2}u = \partial_{x}^{2}u \pm u^{3}\,,
	\quad (t,x)\in\mathbb{R}_{+}\times [0,\pi]
	\,,
\end{equation}
subject to Dirichlet boundary conditions
\begin{equation}
	\label{eq:23.02.14_01a}
	u(t,0)=0=u(t,\pi)\,,
	\quad t\in\mathbb{R}_{+}
	\,,
\end{equation}
and with smooth initial data ($u(0,x)$, $\partial_{t}u(0,x)$). The equation \eqref{eq:23.02.14_01} is Hamiltonian with conserved energy
\begin{equation}
	\label{eq:24.05.08_01}
	E[u] = \int_{0}^{\pi}\left(\frac{1}{2}\left(\partial_{t}u\right)^{2}+\frac{1}{2}\left(\partial_{x}u\right)^{2}\mp\frac{1}{4}u^{4}\right)\diff{x}\,.
\end{equation}
It is easy to show that, for defocusing nonlinearity there are no static solutions other than the zero solution. In contrast, for the focusing equation there exists a countable family of nontrivial static solutions $s_{\kappa}(x)$, $\kappa\in\mathbb{N}$. Excited states, $\kappa>1$, are generated by scaling\footnote{See Eq.~\eqref{eq:24.05.28_01} below.} $s_{\kappa}(x) = \kappa s_{1}(\kappa x)$ of the fundamental solution, given by
\begin{equation}
	\label{eq:24.05.26_01}
	s_{1}(x) = \sqrt{2}\left(3E_{1}/\pi\right)^{1/4}\sn\left(\left(3E_{1}/\pi\right)^{1/4}x,i\right)\,,
	\quad
	E_{1}\equiv E[s_{1}] = 16K(i)^{4}/(3\pi^{3})\,,
\end{equation}
where $\sn(x,k)$ is Jacobi elliptic function and $K(k)$ is the complete elliptic integral of the first kind, see \cite[\href{https://dlmf.nist.gov/22.2}{Sec.~22.2}]{NIST:DLMF} and \cite[\href{https://dlmf.nist.gov/19.2.ii}{Sec.~19.2(ii)}]{NIST:DLMF} respectively. For $\kappa>1$ there is $E_{\kappa}\equiv E[s_{\kappa}]=\kappa^{4}E_{1}$.

In this work we investigate in detail non-trivial time-periodic solutions, i.e. time-dependent solutions to \eqref{eq:23.02.14_01} satisfying
\begin{equation}
	\label{eq:23.02.14_03}
	u(t+T,x) = u(t,x)\,,
	\quad
	x\in (0,\pi)
	\,,
\end{equation}
with $0<T<\infty$.
Let us first consider a linearisation around the trivial solution $u=0$. By dropping the cubic term in \eqref{eq:23.02.14_01}, we can write a general solution as a linear combination of  $\sin{nx}$, the eigenmodes of the operator $\partial_{x}^{2}$, oscillating in time with the corresponding eigenfrequency $n\in\mathbb{N}$, i.e.
\begin{equation}
	\label{eq:}
	u(t,x) = \sum_{n\geq 1} (a_{n}\sin{nt} + b_{n}\cos{nt})\sin{nx}\,,
\end{equation}
where the constants $a_{n},b_{n}\in\mathbb{R}$ are uniquely determined by the initial data. As a side remark, from this it follows that the zero solution is linearly stable. Moreover, we immediately see that there are plenty of time-periodic solutions, in particular a single spatial mode solution $(a_{n}\sin{nt} + b_{n}\cos{nt})\sin{nx}$, $\forall n\in\mathbb{N}$.
A natural question then arises: can any of these solutions be extended as solutions of the nonlinear equation \eqref{eq:23.02.14_01}?

Since the frequencies of the linear problem are natural numbers the analysis of the nonlinear problem naturally encounters resonances. This issue is particularly visible in the perturbative calculation, where resonances spoil the naive perturbative expansion by generating secular terms. Thus, a special care needs to be taken when studying problems with fully resonant spectrum of linear frequencies.

Since the classic paper \cite{lidskii1988periodic}, periodic solutions of cubic wave equation have been an object of thorough investigations. It resulted in a plethora of existence theorems for small amplitude periodic solutions to \eqref{eq:23.02.14_01}-\eqref{eq:23.02.14_01a}. In spite of the fact that these results were obtained by various means, such as averaging techniques \cite{bambusi2001families}, Nash-Moser theorem \cite{Berti.2006, berti2007nonlinear}, variational principle \cite{berti2008cantor}, or Lindstedt series techniques \cite{GM.2004, GMP.2005}, they seem to share a common feature: the sets of frequencies (or amplitudes, depending on the chosen parametrisation) of the found periodic solutions are nowhere dense. The gaps arise from additional restrictions imposed to overcome the small divisor problem. Our main motivation here is to understand the overall structure of the periodic solutions, in particular how these gaps are reflected in it. We discuss both the focusing and defocusing cases.

Using a robust spectrally convergent numerical code, based on Galerkin approximation, we discover a complex structure of time-periodic solutions to \eqref{eq:23.02.14_01}. Parts of this pattern have been earlier observed in \cite{Arioli.2017}. We provide a detailed analysis of these bifurcation structures for the first time.

In addition, for small solutions we use the Poincar\'e-Lindstedt perturbative method and construct families of solutions bifurcating from arbitrary linear eigenfrequency, thus generalising the results presented in \cite{Vernov.1998,Khrustalev.2000,Khrustalev.2001}. We also verify the validity of their ansatz for the solution in the lowest order and expand their perturbative construction by discussing higher orders.
Although the perturbative approach provides a limited number of terms, and the convergence of this series is questionable, the results are accurate enough to test the numerical scheme in the regime of very small solutions.

\section{Perturbative construction}
\label{sec:PerturbativeConstruction}

For a moment, we focus on the defocusing case, i.e., we consider a minus sign in Eq.~\eqref{eq:23.02.14_01}.

\subsection{Ansatz and expansion}
\label{sec:AnsatzAndExpansion}

We introduce a new time coordinate
\begin{equation}
	\label{eq:23.02.15_01}
	\tau = \Omega t
	\,,
\end{equation}
so that the period of a solution is $T=2\pi$. Then the equation \eqref{eq:23.02.14_01} becomes
\begin{equation}
	\label{eq:23.02.15_02}
	\Omega^{2}\partial_{\tau}^{2}u = \partial_{x}^{2}u - u^{3}
	\,.
\end{equation}
In the following, we will focus on solutions that bifurcate from the linearised frequency $\Omega=1$. This is because solutions bifurcating from other frequencies $\Omega=n\in\mathbb{N}$ are easily obtained using the symmetry of the equation \eqref{eq:23.02.15_02}, which says that if $(u(\tau,x), \Omega)$ is a solution, then $(\tilde{u}(\tau,x), \tilde{\Omega})$ given by
\begin{equation}
	\label{eq:24.05.28_01}
	(u(\tau,x),\Omega)\rightarrow(\tilde{u}(\tau,x),\tilde{\Omega})=(n\,u(m\tau,nx),n\Omega/m)\,,
	\quad
	n,m\in\mathbb{N}
	\,.
\end{equation}
is also a solution.

With this in mind we write the perturbative ansatz\footnote{Which takes into account the nature of the nonlinearity, and is equivalent to the one used in \cite{Khrustalev.2001} up to the redefinitions: $\varepsilon\rightarrow\sqrt{\varepsilon_{KV}}$ and $u\rightarrow \varepsilon u_{KV}$.} for a solution 
\begin{equation}
	\label{eq:23.02.15_03}
	u(\tau,x) = \varepsilon u^{(1)}(\tau,x) + \varepsilon^{3} u^{(3)}(\tau,x) + \cdots
	\,,
\end{equation}
and the frequency
\begin{equation}
	\label{eq:23.02.15_04}
	\Omega = 1 + \xi_{2}\varepsilon^{2} + \xi_{4}\varepsilon^{4} + \cdots
	\,,
\end{equation}
where $\xi_{i}$ are real numbers to be fixed at the later stages of the construction. Plugging \eqref{eq:23.02.15_03} and \eqref{eq:23.02.15_04} into \eqref{eq:23.02.15_02} and expanding for small $\varepsilon$ we get a hierarchy of equations, with the first few given explicitly
\begin{align}
	\label{eq:23.02.15_05a}
	\partial_{\tau}^{2}u^{(1)} - \partial_{x}^{2}u^{(1)} &= 0\,,
	\\
	\label{eq:23.02.15_05b}
	\partial_{\tau}^{2}u^{(3)} - \partial_{x}^{2}u^{(3)} &= - 2\xi_{2}\partial_{\tau}^{2}u^{(1)} - \left(u^{(1)}\right)^{3}\,,
	\\
	\label{eq:23.02.15_05c}
	\partial_{\tau}^{2}u^{(5)} - \partial_{x}^{2}u^{(5)} &= - 2\xi_{2}\partial_{\tau}^{2}u^{(3)} - \left(2\xi_{4}+\xi_{2}^{2}\right)\partial_{\tau}^{2}u^{(1)}-3u^{(3)}\left(u^{(1)}\right)^{2}
	\,.
\end{align}
To fix the ambiguity of the definition of the small parameter $\varepsilon$ we require the $\sin{\tau}\sin{x}$ mode to be present only at the first perturbative order, i.e. we set
\begin{equation}
	\label{eq:24.05.12_03}
	\int_{0}^{\pi}\diff{\tau}\int_{0}^{\pi}\diff{x}\, u^{(2j+1)}(\tau,x)\sin{\tau}\sin{x} = 0
	\,,
	\quad
	j\geq 1
	\,.
\end{equation}

\subsection{First order}

At the lowest order, we consider solutions of the equation \eqref{eq:23.02.15_05a} in the form
\begin{equation}
	\label{eq:23.02.15_06}
	u^{(1)}(\tau,x) = \sum_{n=1}^{\infty}a_{n}\sin n\tau \sin nx
	\,,
\end{equation}
with the expansion coefficients $a_{n}$ not fixed at this stage. As will be shown shortly, the solution needs to be a specific combination of an infinite number of modes, contrary to other models where it suffices to take at the lowest order a single eigenmode corresponding to the bifurcating frequency \cite{Maliborski.2013,Maliborski.2015}.

It may seem that we restricted ourselves by fixing the phases of each mode. However, as it will be clear from the following construction, the solutions do actually synchronise their phases, and this synchronisation holds at each perturbative order.

\subsection{Third order}
At the next order, we have to solve the inhomogeneous equation \eqref{eq:23.02.15_05b}. Since we are looking for bounded-in-time solutions the resonant terms need to vanish. This requirement imposes conditions on $\xi_{2}$ and $a_{n}$. Explicitly, projecting the source term of \eqref{eq:23.02.15_05b} onto $\sin{j\tau}\sin{jx}$ we get
\begin{equation}
	\label{eq:23.02.15_08}
	\int_{0}^{\pi}\diff{\tau}\int_{0}^{\pi}\diff{x}\left( 2\xi_{2}\partial_{\tau}^{2}u^{(1)}(\tau,x) + \left(u^{(1)}(\tau,x)\right)^{3}\right)\sin{j\tau}\sin{jx} = 0
	\,,
\end{equation}
for $j\in\mathbb{N}$. We have the following
\begin{lemma}
Let $k=0.451075598810\ldots$ be a unique root of
\begin{equation}
	\label{eq:24.06.03_02}
	K(k)\left(6E(k)+(8k^{2}-7)K(k)\right) = 0
	\,,
\end{equation}
on $[0,1]$, with $E(k)$ denoting here the complete elliptic integral of the second kind \cite[\href{https://dlmf.nist.gov/19.2.ii}{Sec.~19.2(ii)}]{NIST:DLMF}. If $\xi_{2}=1/256$ and
\begin{equation}
	\label{eq:24.06.03_01}
	u^{(1)}(\tau,x)=\frac{k}{2\gamma}\left[\cn\left(\frac{4}{\gamma}(\tau-x),k\right)-\cn\left(\frac{4}{\gamma}(\tau+x),k\right)\right]
	\,,
	\quad
	\gamma=2\pi/K(k)
	\,,
\end{equation}
where $\cn(x,k)$ is a Jacobi elliptic function \cite[\href{https://dlmf.nist.gov/22.2}{Sec.~22.2}]{NIST:DLMF}, then the resonant terms \eqref{eq:23.02.15_08} vanish.
\end{lemma}

\begin{proof}
As the first step we use \eqref{eq:23.02.15_06} to rewrite the condition \eqref{eq:23.02.15_08} as
\begin{equation}
	\label{eq:23.02.15_09}
	2\left(\frac{\pi}{2}\right)^{2}\xi_{2}j^{2}a_{j} + \int_{0}^{\pi}\diff{\tau}\int_{0}^{\pi}\diff{x}\left(\sum_{n=1}^{\infty}a_{n}\sin{n\tau}\sin{nx}\right)^{3}\sin{j\tau}\sin{jx} = 0
	\,.
\end{equation}
Next, the product rule for the $\sin$ function
\begin{multline}
	\label{eq:23.02.15_10}
	\sin{nx}\sin{mx}\sin{lx} = 
	\\
	\frac{1}{4}\left[
	\sin(-n+m+l)x + \sin(n-m+l)x + \sin(n+m-l)x - \sin(n+m+l)x
	\right]
\end{multline}
after a tedious calculation leads to
\begin{multline}
	\label{eq:23.02.15_11}
	- 32\xi_{2}j^{2}a_{j} + 
	\left(3\sum_{l=1}^\infty\sum_{m=1}^\infty a_{l+m-j}a_{l}a_{m} + 3\sum_{l=1}^\infty \sum_{m=1}^\infty a_{l+m+j}a_{l}a_{m} 
	\right.
	\\
	\left.
	+ \sum_{l=1}^{j-1} \sum_{m=1}^{j-l-1} a_{-l-m+j}a_{l}a_{m} + 6\sum_{m=1}^\infty a_{m}^{2}a_{j}\right) = 0
	\,.
\end{multline}
Thus, we are left with an infinite system of algebraic equations.

Following \cite{Vernov.1998,Khrustalev.2000,Khrustalev.2001}, which in turn was motivated by numerical experiments, we show that for $\xi_{2}$ and $k$ specified in the statement, the system \eqref{eq:23.02.15_11} is solved by the Fourier coefficients $f_n$ of the Jacobi elliptic function
\begin{equation}
	\label{eq:23.02.24_01}
	\cn(x,k) = \frac{\gamma}{k}\sum_{n=1}^{\infty}f_{2n-1}\cos\left((2n-1)\frac{\gamma}{4}x\right)
	\,,
\end{equation}
where the coefficients $f_{n}$ are given by
\begin{align}
	\label{eq:23.02.24_03}
	f_n=\begin{cases}
	\displaystyle \frac{q^{n/2}}{1+q^n},\qquad & \mbox{if $n$ is odd}\, ,
	\\
	\displaystyle 0 \qquad & \mbox{if $n$ is even}\, .
	\end{cases}
\end{align}
where $q$ is a nome function in $k$ given by \cite[\href{https://dlmf.nist.gov/22.2.E1}{(22.2.1)}]{NIST:DLMF}  
\begin{align}\label{eqn:nome}
	q=\exp\left(-\pi \frac{K\left(\sqrt{1-k^2}\right)}{K(k)}\right) \,.
\end{align}
Now, let us expand the third power of $\cn(x,k)$ as a Fourier series. The expression
\begin{multline}
	\label{eq:23.02.24_04}
	\cn(x,k)^{3} = 
	\\
	\left(\frac{\gamma}{k}\right)^{3}\sum_{n,m,l=1}^\infty f_{2n-1}f_{2m-1}f_{2l-1} \cos\left((2n-1)\frac{\gamma}{4}x\right)\cos\left((2m-1)\frac{\gamma}{4}x\right)\cos\left((2l-1)\frac{\gamma}{4}x\right)
	\,,
\end{multline}
can be reduced with the use of the following identity
\begin{multline}
	\label{eq:3coss}
	\cos nx\, \cos mx \, \cos lx = \\
	\frac{1}{4}\left(\cos(-n+m+l)x + \cos(n-m+l)x + \cos(n+m-l)x + \cos(n+m+l)x\right) \,.
\end{multline}
Then the orthogonality relation for $\cos$ functions lets us write \eqref{eq:23.02.24_04} as
\begin{equation}
	\label{eq:23.02.24_06}
	\cn(x,k)^{3} = \frac{\gamma^{3}}{4k^{3}}\sum_{n=1}^{\infty}F_{n}(f)\cos\left(n \frac{\gamma}{4} x\right)
	\,,
\end{equation}
where the form of $F_n$ depends on the parity of its index:
\begin{multline}
	\label{eq:23.02.27_01}
	F_{2n-1}(f) = 3 \sum_{l=1}^\infty \sum_{m=1}^\infty\left(
	f_{2l-1}f_{2m-1}f_{2(l+m-n)-1} + f_{2l-1}f_{2m-1}f_{2(l+m+n)-3}\right)
	\\
	+ \sum_{l=1}^n \sum_{m=1}^{n-l} f_{2l-1}f_{2m-1}f_{2(-l-m+n)+1}
	\,,
\end{multline}
\begin{equation}
	\label{eq:23.02.27_02}
	F_{2n}(f) = 0
	\,.
\end{equation}
Using \eqref{eq:23.02.24_06} we find that the resonant equation \eqref{eq:23.02.15_11}, when evaluated at $a_{j}=f_{j}$ for $j$ odd, cf.~\eqref{eq:23.02.24_03}, takes the form
\begin{equation}
	\label{eq:23.02.24_07}
	2\xi_{2}j^{2}f_{j} - \frac{1}{16}\left(F_{j}(f) + 6\sum_{m=1}^{\infty}f_{m}^{2}f_{j}\right)=0
	\,.
\end{equation}
From the differential equation satisfied by $\cn(x,k)$
\begin{equation}
	\label{eq:23.05.15_01}
	\frac{\diff^{2}}{\diff{x}^{2}}\cn(x,k) = (2k^{2}-1)\cn(x,k)-2k^{2}\cn(x,k)^{3}
	\,,
\end{equation}
we have
\begin{equation}
	\label{eq:23.02.24_08}
	F_{j}(f) = \left(\frac{2(2k^{2}-1)}{\gamma} + \frac{j^2}{8}\right)f_{j}
	\,,
\end{equation}
which allows us to rewrite \eqref{eq:23.02.24_07} as follows
\begin{equation}
	\label{eq:23.02.24_09}
	0 = f_{j}\left[2\xi_{2}j^{2} - \frac{1}{16}\left(\frac{2(2k^{2}-1)}{\gamma} + \frac{j^2}{8} + 6\sum_{m=1}^{\infty}f_{m}^{2}\right)\right]
	\,.
\end{equation}
Comparing coefficients of a different power of $j$ we get
\begin{align}
	\label{eq:23.02.24_10}
	\xi_{2} &= \frac{1}{256}\,,
	\\
	\label{eq:23.02.24_11}
	\frac{2k^{2}-1}{\gamma} + 3 \sum_{m=1}^{\infty}f_{m}^{2} &= 0
	\,.
\end{align}
Since $\gamma$ and $f_m$ are functions of $k$, the second equation is an implicit condition for the elliptic modulus. As we prove in Appendix \ref{sec:DieckmannIdentity}, Eq.~\eqref{eq:23.02.24_11} has a unique solution $k=0.451075598810\ldots$. Hence, there exists a unique value of $k$ such that $f_n$ solve \eqref{eq:23.02.15_11}. 

As the last step, we simply unravel the explicit form of $u^{(1)}$ by changing $a_n$ in \eqref{eq:23.02.15_06} to $f_n$. Then, elementary transformations lead to
\begin{align}
	\label{eq:23.02.28_01}
	u^{(1)}(\tau,x)&=\sum_{n=1}^\infty f_n \sin{n\tau}\sin{n x}=\sum_{n=1}^\infty f_{2n-1} \sin{(2n-1)\tau}\sin{(2n-1) x}=\\
	&=\frac{1}{2}\sum_{n=1}^\infty f_{2n-1} \left[\cos{(2n-1)(\tau-x)}-\cos{(2n-1)(\tau+x)}\right]=\\
	&=\frac{k}{2\gamma}\left[\cn\left(\frac{4}{\gamma}(\tau-x),k\right)-\cn\left(\frac{4}{\gamma}(\tau+x),k\right)\right]\,,
\end{align}
where we have used \eqref{eq:23.02.24_01}.
\end{proof}

At this point the resonant terms in \eqref{eq:23.02.15_05c} have been removed and the lowest order solution \eqref{eq:23.02.15_06} is fixed uniquely. Next, using the ansatz
\begin{equation}
	\label{eq:23.02.15_000}
	u^{(3)}(\tau,x) = \sum_{j,k=0}^{\infty}b_{jk}\sin{j\tau}\sin{kx}
	\,,
\end{equation}
we can find the explicit solution for the off diagonal coefficients $b_{jk}$ to be
\begin{equation}
	\label{eq:23.02.15_0000}
	b_{jk} = \frac{B_{jk}}{k^{2}-j^{2}}
	\,,
	\quad
	j\neq k\, ,
\end{equation}
whereas $b_{jj}$ are left undetermined. The numbers $B_{jk}$ are simply non-diagonal Fourier coefficients of $(u^{(1)})^{3}$. They can be written down explicitly using the special form of $f_j$ (\ref{eq:23.02.24_03}). The fact that $f_j$ are nonzero only for odd $j$ leads to $B_{jk}$ also being nonzero only for odd indices. Then, after resummations presented in the Appendix~\ref{sec:DerivationOfExplicitExpressionsForDjk} one gets the following formula
\begin{equation}
	\label{eq:23.03.14_001}
	B_{jk}=\begin{cases}
	\displaystyle -\frac{3}{128}\frac{j-k}{\sinh\left(\frac{j}{2}\ln q\right) - \sinh\left(\frac{k}{2}\ln q\right)}\, ,\qquad& \mbox{if $j-k \equiv 0 \mod 4$\,,}
	\vspace{2ex}
	\\
	\displaystyle \frac{3}{128}\frac{j+k}{\sinh\left(\frac{j}{2}\ln q\right) + \sinh\left(\frac{k}{2}\ln q\right)}\, ,\qquad& \mbox{if $j-k \equiv 2 \mod 4$\,.}
	\end{cases}
\end{equation}
Note that for the off diagonal coefficients there is $b_{jk}=-b_{kj}$.

\subsection{Higher orders}
\label{sec:higher_orders}

We look for the resonant terms in the Eq.~\eqref{eq:23.02.15_05c}. If we set $b_{jj}=0$ in \eqref{eq:23.02.15_000} then, since $b_{jk}=-b_{kj}$, the last term on the right hand side of \eqref{eq:23.02.15_05c} does not contain diagonal terms, see Appendix \ref{sec:ResonanceConditionsAtThirdOrder}. Consequently, to remove the resonant terms generated by $u^{(1)}$ it is necessary to set
\begin{equation}
	\label{eq:24.05.09_01}
	\xi_{4} = - \frac{1}{2}\xi_{2}^{2}\,,
\end{equation}
Finally, the solution $u^{(5)}$ can be written, similarly as at the previous order, as
\begin{equation}
	\label{eq:24.05.09_02}
	u^{(5)}(\tau,x) = \sum_{j,k=0}^{\infty}c_{jk}\sin{j\tau}\sin{kx}
	\,,
\end{equation}
with
\begin{equation}
	\label{eq:24.05.09_03}
	c_{jk} = \frac{1}{k^{2}-j^{2}}\left(2\xi_{2}j^{2}\frac{B_{jk}}{k^{2}-j^{2}}-3C_{jk}\right)
	\,,
	\quad
	j\neq k
\end{equation}
where $C_{jk}$ are the off diagonal Fourier coefficients of $u^{(3)}\left(u^{(1)}\right)^2$ and $c_{jj}$ are to be set at the next order. In principle, this procedure could be continued to an arbitrary order, with the diagonal terms and frequency expansion coefficients serving as free parameters to eliminate resonances. However, since the calculations become increasingly complex at this stage, we stop the procedure without providing the explicit formula for $C_{jk}$. In principle this gives only a formal solution and based on the existing results \cite{GMP.2005} we do not expect the series to converge. In summary, this approach yields an approximation to a time-periodic solution up to $\varepsilon^4$ order (included). 

\subsection{Focusing nonlinearity}
\label{sec:FocusingNonlinearity}
The sign of the nonlinearity in Eq.~\eqref{eq:23.02.14_01} has a minor impact on the presented results. In the focusing case, the introduction of the rescaled time $\tau=\Omega t$ leads to $\Omega^{2}\partial_{\tau}^{2}u = \partial_{x}^{2}u + u^{3}$. Then, this equation at the lowest perturbative order \eqref{eq:23.02.15_05a} is exactly the same, while the next order differs only in the sign before the cubic term
\begin{equation}
	\label{eq:23.04.13_01}
	\partial_{\tau}^{2}u^{(3)} - \partial_{x}^{2}u^{(3)} = -2\xi_{2}\partial_{\tau}^{2}u^{(1)} + \left(u^{(1)}\right)^{3}
	\,,
\end{equation}
cf. \eqref{eq:23.02.15_05b}. Since we are interested in such $u^{(1)}$ that the diagonal terms on the right hand side vanish, one can just define $\tilde{\xi}_2=-\xi_2$ and insert it into the equation. This change reduces the problem to the one studied in the defocusing case above. As a result, we get the same functional from for $u^{(1)}$, while now $\xi_2=-1/256$. Also the next order can be easily covered as once again we can set the diagonal terms in $u^{(3)}$ to zero. The non-diagonal terms are given up to sign by the same expressions as in the defocusing case, see \eqref{eq:23.03.14_001}. At the next order, following the same argument, we get $\xi_4=-\xi_2^2/2$. For higher orders, the perturbative equations become more complex, resulting in no simple correspondence between focusing and defocusing solutions, as it was the case for the lowest orders.

Alternatively, one can explore the symmetry that transforms a time-periodic solution $(u(\tau,x),\Omega)$ of the defocusing equation  into a solution $(\tilde{u}(\tau,x),\tilde{\Omega})$ of the focusing equation
\begin{equation}
	\label{eq:24.05.13_01}
	(\tilde{u}(\tau,x),\tilde{\Omega}) = (\Omega^{-1}u(x,\tau),\Omega^{-1})
	\,.
\end{equation}
However, this yields a solution with a different parametrisation, in which \eqref{eq:24.05.12_03} is no longer satisfied.

\section{Numerical construction}
\label{sec:NumericalConstruction}

The following numerical approach applies to both signs of the nonlinearity.
However, in the subsequent subsections, we focus on the defocusing case,
while the results for the focusing nonlinearity are discussed at the end of this section.

\subsection{Numerical approach}
\label{sec:NumericalApproach}
To find time-periodic solutions of \eqref{eq:23.02.15_02}, we employ a Galerkin method with numerical integration, also referred to as the collocation method in the weak form \cite{Shen.2011}. We look for solutions in a finite-dimensional subspace of a Hilbert space\footnote{The use of the $\cos$ functions in the $\tau$ direction, as opposed to $\sin$ implemented in the perturbative construction, is merely for convenience and corresponds to a particular choice of phase for the solutions. Additionally, considering the functional form of the solution, as suggested by the perturbative results of the previous section, we only consider odd Fourier modes.}
\begin{equation}
	\label{eq:24.05.09_04}
	\textrm{span}\left\{\left.\cos(2j+1)\tau \sin(2k+1)x\,\right|\ j=0,\ldots,N_{\tau}-1\,,\ k=0,\ldots,N_{x}-1\right\}
	\,,
\end{equation}
i.e. we approximate a solution $u(\tau,u)$ by a finite Fourier series
\begin{equation}
	\label{eq:24.05.13_02}
	u_{N_{\tau},N_{x}}(\tau,x) = \sum_{j=0}^{N_{\tau}-1}\sum_{k=0}^{N_{x}-1}a_{jk}\cos(2j+1)\tau \sin(2k+1)x
	\,,
\end{equation}
On the product space $(\tau,x)\in[0,2\pi]\times[0,\pi]$ we define the collocation points
\begin{equation}
	\label{eq:24.05.13_03}
	\tau_{j} = \frac{\pi(j+1/2)}{2N_{\tau}+1}\,,
	\quad
	j=0,\ldots,N_{\tau}-1\,,
	\quad
	x_{k} = \frac{\pi(k+1)}{2N_{x}+1}
	\,,
	\quad
	k=0,\ldots,N_{x}-1
	\,,
\end{equation}
$N_{\tau},N_{x}>0$ and the corresponding discrete inner products 
\begin{equation}
	\label{eq:24.05.13_04}
	\langle f, g\rangle_{\tau} = \sum_{j=0}^{N_{\tau}-1}f(\tau_{j})g(\tau_{j})w_{j}
	\,,
	\quad
	w_{j} = \frac{2\pi}{2N_{\tau}+1}	
	\,,
\end{equation}
and
\begin{equation}
	\label{eq:24.05.13_05}
	\langle f, g\rangle_{x} = \sum_{k=0}^{N_{x}-1}f(x_{k})g(x_{k})\varpi_{k}
	\,,
	\quad
	\varpi_{k} = \frac{2\pi}{2N_{x}+1}	
	\,,
\end{equation}
which approximate the respective continuous inner products.

Plugging \eqref{eq:24.05.13_02} into Eq.~\eqref{eq:23.02.15_02} and requiring the residual to be orthogonal to $\cos(2m+1)\tau\sin(2n+1)x$, for $m=0,\ldots,N_{\tau}-1$ and $n=0,\ldots,N_{x}-1$, we get
\begin{equation}
	\label{eq:24.05.13_06}
	\left(\frac{\pi}{2}\right)^{2}\left(-\Omega^{2}(2m+1)^{2}+(2n+1)^{2}\right)a_{mn} + a^{(3)}_{mn} = 0
	\,,
\end{equation}
where
\begin{equation}
	\label{eq:24.05.13_07}
	a^{(3)}_{mn} = \sum_{j=0}^{\tilde{N}_{\tau}-1}\sum_{k=0}^{\tilde{N}_{x}-1}\left(u_{N_{\tau},N_{x}}(\tilde{\tau}_{j},\tilde{x}_{k})\right)^{3}\cos(2m+1)\tilde{\tau}_{j}\sin(2n+1)\tilde{x}_{k}\,\tilde{w}_{j}\tilde{\varpi}_{k}
	\,,
\end{equation}
are the Fourier coefficients of the cubic term, computed in the physical space. As such these coefficients are homogeneous polynomials of degree 3 in $a_{mn}$'s. Note that to eliminate the aliasing errors in \eqref{eq:24.05.13_07} we use the discrete inner products with $\tilde{N}_{\tau}=3N_{\tau}-1$ and $\tilde{N}_{x}=3N_{x}-1$ quadrature points $(\tilde{\tau}_{j},\tilde{w}_{j})$ and $(\tilde{x}_{k},\tilde{\varpi}_{k})$ respectively. Therefore, the integral in \eqref{eq:24.05.13_07} is exact and this formulation is equivalent to the Galerkin approximation. However, because it belongs to a class of pseudo-spectral methods, it has much lower computational complexity than the traditional Galerkin approach, which is its main advantage. Although the exact Jacobian for this system of equations can be derived relatively easily, its computation is complex and resource-intensive. Therefore, we compute the Jacobian matrix numerically, using a finite difference approximation.

For further reference, the energy of the solution \eqref{eq:24.05.13_02} can be obtained by inserting the truncated series \eqref{eq:24.05.13_02} into \eqref{eq:24.05.08_01} and evaluating the integrand at $\tau = \pi/2$. The result is
\begin{equation}
	\label{eq:24.05.13_08}
	E = \frac{\pi}{4}\Omega^{2} \sum_{n=0}^{N_{x}-1}\left(\sum_{m=0}^{N_{\tau}-1}(-1)^{m}(2m+1)a_{mn}\right)^{2}
	\,.
\end{equation}

The described discretisation turns the problem of finding time-periodic solutions into solving an algebraic system of the form 
\begin{equation}
	\label{eq:24.05.12_01}
	F(\textbf{a},\Omega)=0\,,
	\quad
	F:\mathbb{R}^{n}\times \mathbb{R} \rightarrow \mathbb{R}\,,
	\quad
	n=N_{\tau}N_{x}
	\,,
\end{equation}
where $\textbf{a}$ is a vector storing the Fourier coefficients \eqref{eq:24.05.13_02}. Based on current and previous results \cite{Maliborski.2015}, we expect that there exists a continuous family of solutions, a path of solutions, that bifurcates from $(\textbf{a},\Omega)=(0,1)$. However, as one may encounter other bifurcation points along the path, a natural continuation, with $\Omega$ as a parameter, will fail at such locations.
To overcome this, we use a path following strategy, namely the pseudo-arclength continuation method \cite{Keller.1987}, which we briefly review. Given a solution $(\textbf{a}^{0},\Omega^{0})$ on a solution path $\mathbb{R} \supset I \ni s\rightarrow (\textbf{a}(s),\Omega(s))$ and a unit tangent vector $(\dot{\textbf{a}}^{0},\dot{\Omega}^{0})$ we look for a solution $(\textbf{a}^{1},\Omega^{1})$ to
\begin{equation}
	\label{eq:24.05.12_02}
	F(\textbf{a}^{1},\Omega^{1}) = 0\,,
	\quad
	G(\textbf{a}^{1},\Omega^{1}) \equiv \left(\textbf{a}^{1}-\textbf{a}^{0}\right)\cdot \left(\dot{\textbf{a}}^{0}\right)^{\transpose} + \left(\Omega^{1}-\Omega^{0}\right)\dot{\Omega}^{0} - \Delta s = 0
	\,.
\end{equation}
The system of equations \eqref{eq:24.05.12_02} is solved using the Newton's scheme, while the extra equation guarantees that we are able to find solutions even when the Jacobian of $F$ becomes degenerate or ill-defined. The step \eqref{eq:24.05.12_02} is then iterated as we follow the solution curve. The initial tangent vector is computed from the condition $\frac{\diff{}}{\diff{s}}F(\textbf{a}(s),\Omega(s))=0$. The adjustable parameter $\Delta s > 0$ controls the step-size along the solution curve. For more details see \cite{Keller.1987}.

The use of the pseudo-arclength continuation method was crucial for efficiently exploring large solutions. However, for solutions of very small amplitude, a natural continuation was sufficient. We discuss such solutions first.

\subsection{Small solutions}
\label{sec:SmallSolutions}

\begin{figure}[!t]
	\centering
	\includegraphics[width=1.00\textwidth]{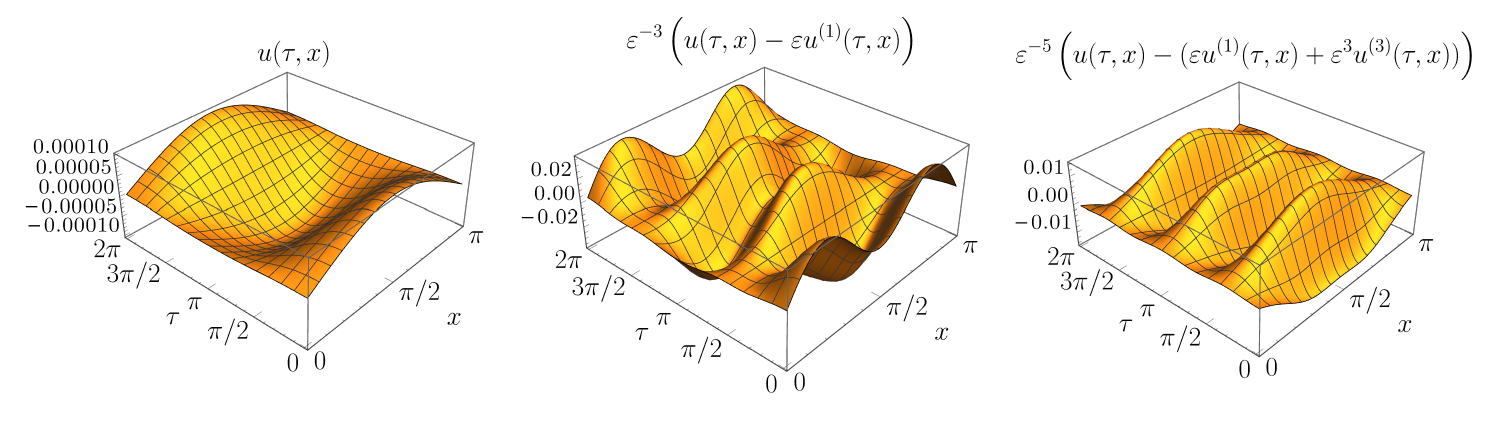}
	\caption{The spatio-temporal plot of the numerical solution $u(\tau,x)$ (left panel) of size corresponding to $\varepsilon=10^{-4}$. The middle and right panels highlight the difference between the numerical data and first- and third-order accurate approximations respectively. Note that these differences have been rescaled by the appropriate powers of $\varepsilon$.}
	\label{fig:SmallProfiles}
\end{figure}
\begin{figure}[!t]
	\centering
	\includegraphics[width=0.475\textwidth]{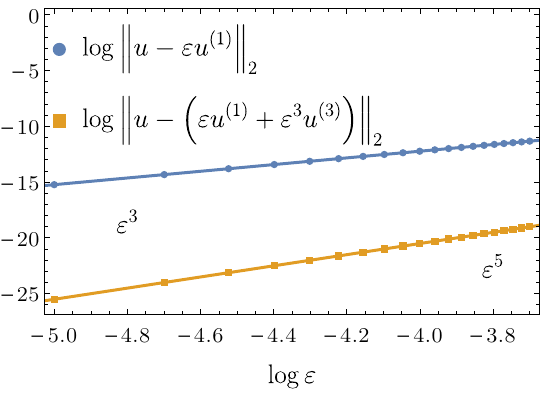}
	\caption{Scaling of the $L^2$ norm of a difference between numerical and perturbative solutions with respect to the solution amplitude $\varepsilon$. The lines are the fits to the data (points) with slopes corresponding to \mbox{$\|u-\varepsilon u^{(1)}\|_{2}\sim\varepsilon^{3}$} and $\|u-\left(\varepsilon u^{(1)}+\varepsilon^{3} u^{(3)}\right)\|_{2}\sim\varepsilon^{5}$ respectively.}
	\label{fig:SmallProfilesError}
\end{figure}

Here we compare perturbative and numerical results for small amplitude. Since we have only two leading order terms at our disposal, we focus on values of the expansion parameter $0<\varepsilon\leq 2\times 10^{-4}$. To achieve accurate results, we solve the system of equations \eqref{eq:24.05.13_06} for relatively high truncation $N_{\tau}=N_{x}=32$ and use extended precision with 64 significant digits.

\begin{figure}[!t]
	\centering
	\includegraphics[width=1.00\textwidth]{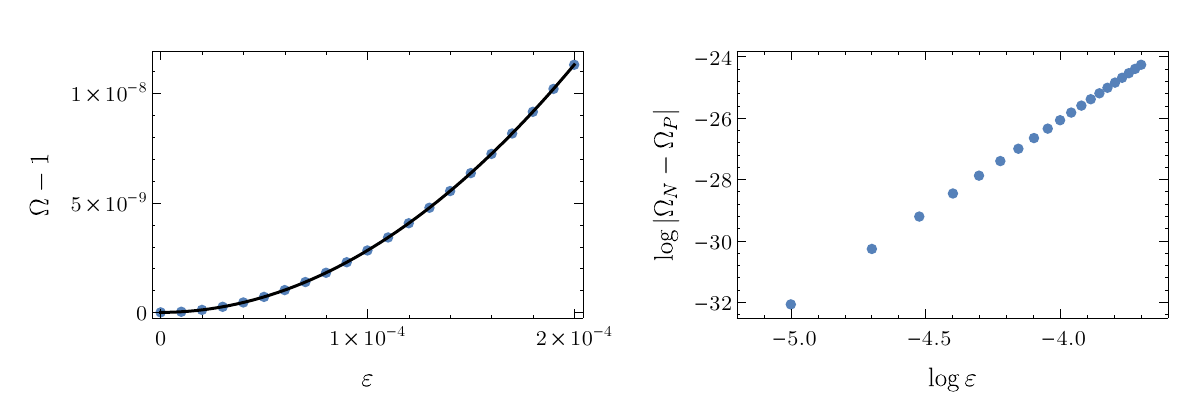}
	\caption{Comparison of the frequencies of small time-periodic solutions calculated numerically and using the perturbative method, see \eqref{eq:23.02.15_04}, \eqref{eq:23.02.24_10}, and \eqref{eq:24.05.09_01}. The left panel shows the frequency $\Omega$ as a function of the amplitude $\varepsilon$, while the right panel displays the absolute error between the numerical $\Omega_{N}$ and perturbative $\Omega_{P}$ result. The data points form a line with a slope of $6$, which is consistent with the perturbative expansion \mbox{$|\Omega_{N}-\Omega_{P}|\sim\varepsilon^{6}$}.}
	\label{fig:SmallFrequency}
\end{figure}

First, we examine solutions with a fixed value of the parameter $\varepsilon$. Specifically, we  compare profiles of numerically calculated solutions with the perturbative series.\footnote{After adjusting the phase difference, cf. \eqref{eq:24.05.13_02} and \eqref{eq:23.02.15_06}.} For small values of the parameter, the agreement between the two approaches is not visible on the scale of the plot presented in Fig.~\ref{fig:SmallProfiles}. Moreover, the difference scaled by appropriate powers of $\varepsilon$ remains small and consistent with the order of perturbative expansion. This makes us confident about both results. Additionally, we investigate how the difference between the two approximations behaves when $\varepsilon$ varies. The dependence of the $L_{2}$ norm of this difference as a function of $\varepsilon$ is presented in Fig.~\ref{fig:SmallProfilesError}. The error terms scale according to the analytical predictions.

\begin{figure}[!t]
	\centering
	\includegraphics[width=0.93\textwidth]{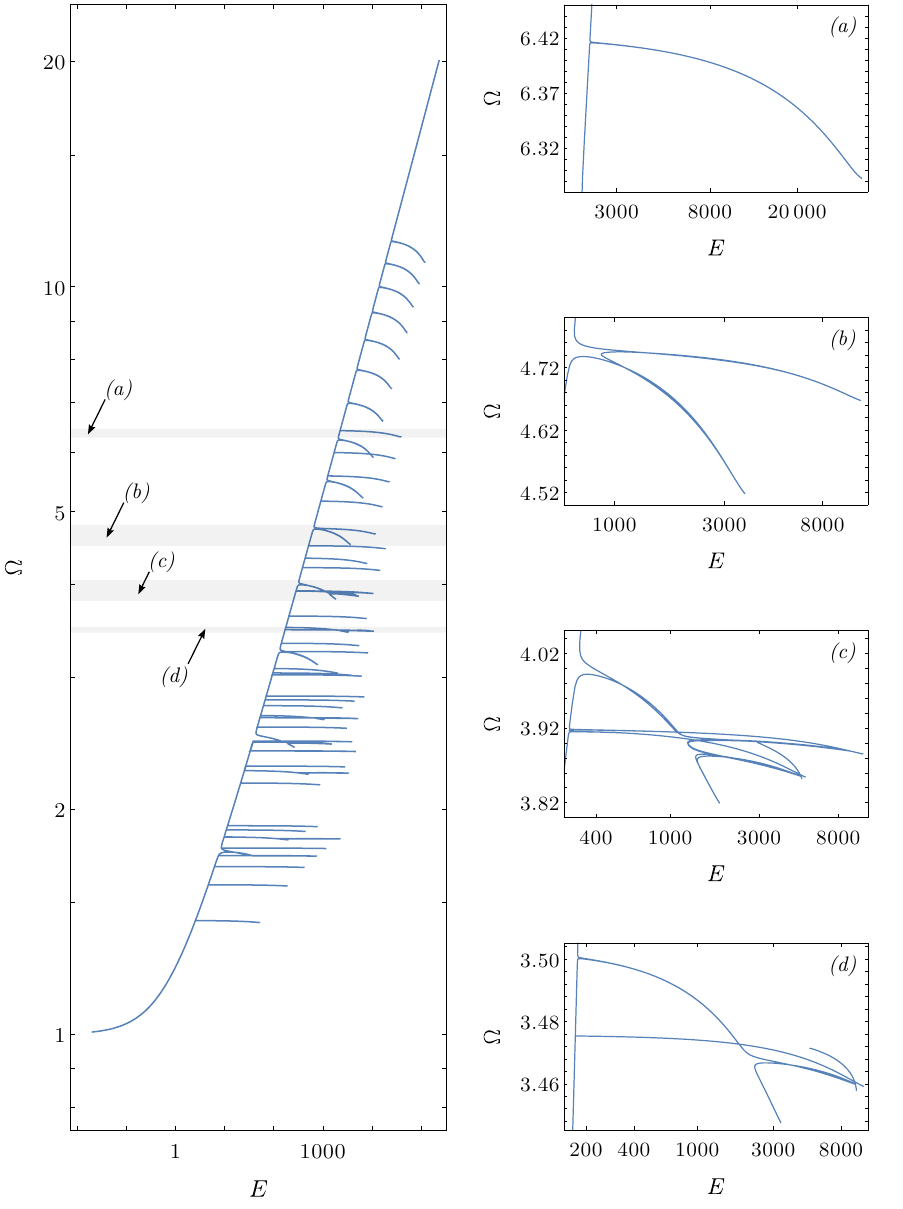}
	\caption{The structure of time-periodic solutions bifurcating from the fundamental frequency $\Omega=1$ found with the help of the arc-length algorithm, see description around Eq.~\eqref{eq:24.05.12_02}. We present results for $N_{\tau}=N_{x}=16$ mode truncation in the Galerkin system \eqref{eq:24.05.13_06}-\eqref{eq:24.05.13_07}. Note the usage of the log-log scale on the full $\Omega-E$ diagram. The right panels zoom in on the details of the left plot (indicated by shaded regions).}
	\label{fig:EOmegaFull}
\end{figure}

Furthermore, we compare the frequencies of small solutions, as shown in Fig.~\ref{fig:SmallFrequency}, where we also observe consistency between the two approaches. The solution of the Galerkin system accurately approximates the biquadratic polynomial \eqref{eq:23.02.15_04} with coefficients given in \eqref{eq:23.02.24_10} and \eqref{eq:24.05.09_01}.

\subsection{Large solutions}
\label{sec:LargeSolutions}

For larger solutions we rely purely on the Galerkin approach, whose quality and robustness is discussed at the end of this subsection, and use the continuation method to follow the path of solutions starting roughly at the rightmost point on Fig.~\ref{fig:SmallFrequency}. Since for such solutions the natural continuation with $\Omega$, used before is not practical, particularly if we encounter folds or bifurcation points \cite{Keller.1987}, we use the pseudo-arclength parametrisation discussed in Sec.~\ref{sec:NumericalApproach}.

For concreteness, we focus on solutions of the system \eqref{eq:24.05.13_06}-\eqref{eq:24.05.13_07} with $N_{\tau}=N_{x}=16$ mode truncation. Later, we comment on how the results behave as $N_{\tau}$ and $N_{x}$ increase. To get rid of the parametrisation ambiguity, we present solutions on the frequency-energy plot, see Eq.~\eqref{eq:24.05.13_08}. A quick look at Fig.~\ref{fig:EOmegaFull} reveals that the solution path is not a simple continuation of the graph shown in Fig.~\ref{fig:SmallFrequency}. Instead, the curve exhibits twists and turns, forming an intricate pattern on the $\Omega-E$ diagram.

However, it is notable that there is a main (diagonal) part of the diagram that resembles what could be considered a natural continuation of the plot in Fig.~\ref{fig:SmallFrequency}. We refer to this part of the diagram as the \textit{primary branch}. Along the primary branch the solution resembles the fundamental mode $\cos{\tau}\sin{x}$, albeit with a significant admixture of higher harmonics. We expect this curve can be continued indefinitely as $\Omega \rightarrow \infty$ without additional features appearing beyond those shown.

\begin{figure}[!t]
	\centering
	\includegraphics[width=0.98\textwidth]{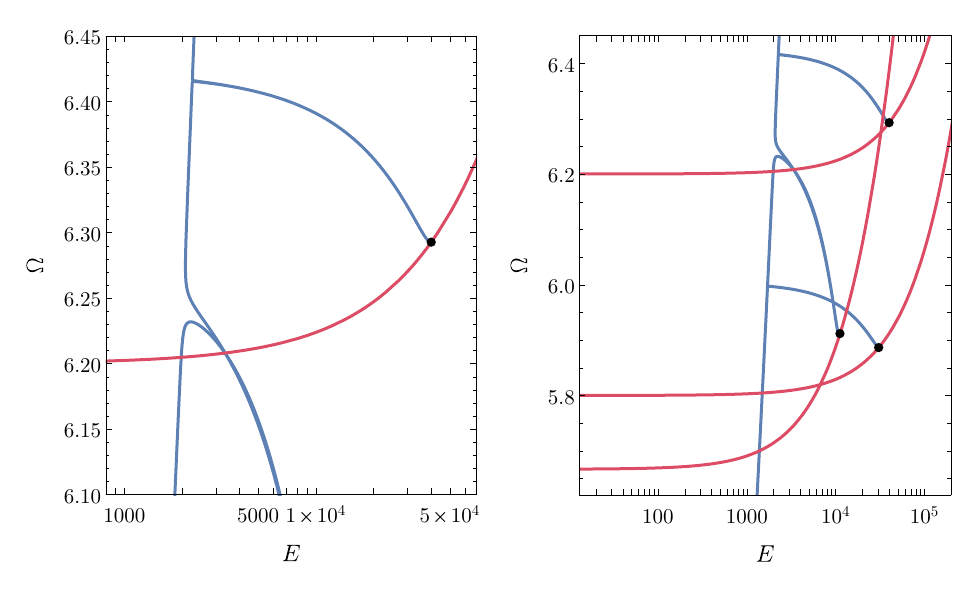}
	\caption{Left: The bifurcation at one of the solution branches, see top right panel on Fig.~\ref{fig:EOmegaFull}. The primary (blue) and secondary (red) branches meet at the bifurcation point denoted by a dot (other intersections of the blue line are an artefacts of projecting the solutions onto the energy plane). The secondary branch (approximately) coincides with an appropriately rescaled primary branch, see Eq.~\eqref{eq:24.05.28_01}. Right: Such bifurcations occur at each of the turning points visible on Fig.~\ref{fig:EOmegaFull}. Here, we observe the cases leading to rescaled solutions bifurcating from frequencies $\Omega=31/5$, $29/5$, and $17/3$.}
	\label{fig:EOmegaBifurcation}
\end{figure}

Emerging from the main part are many almost vertical branches of finite length. Some of these attached branches are simple, while others have a more complicated form. A few examples are shown in Fig.~\ref{fig:EOmegaFull}. The intersections of the branches, clearly visible in the zoomed-in panels, are an artefact of projecting the solution $(\mathbf{a},\Omega)$ of the Galerkin system onto the energy plane. It is important to note that the solution path forms a continuous curve in the $(\mathbf{a},\Omega)$ space as expected from a solution of a finite-dimensional algebraic system.

\begin{figure}[!t]
	\centering
	\includegraphics[width=0.98\textwidth]{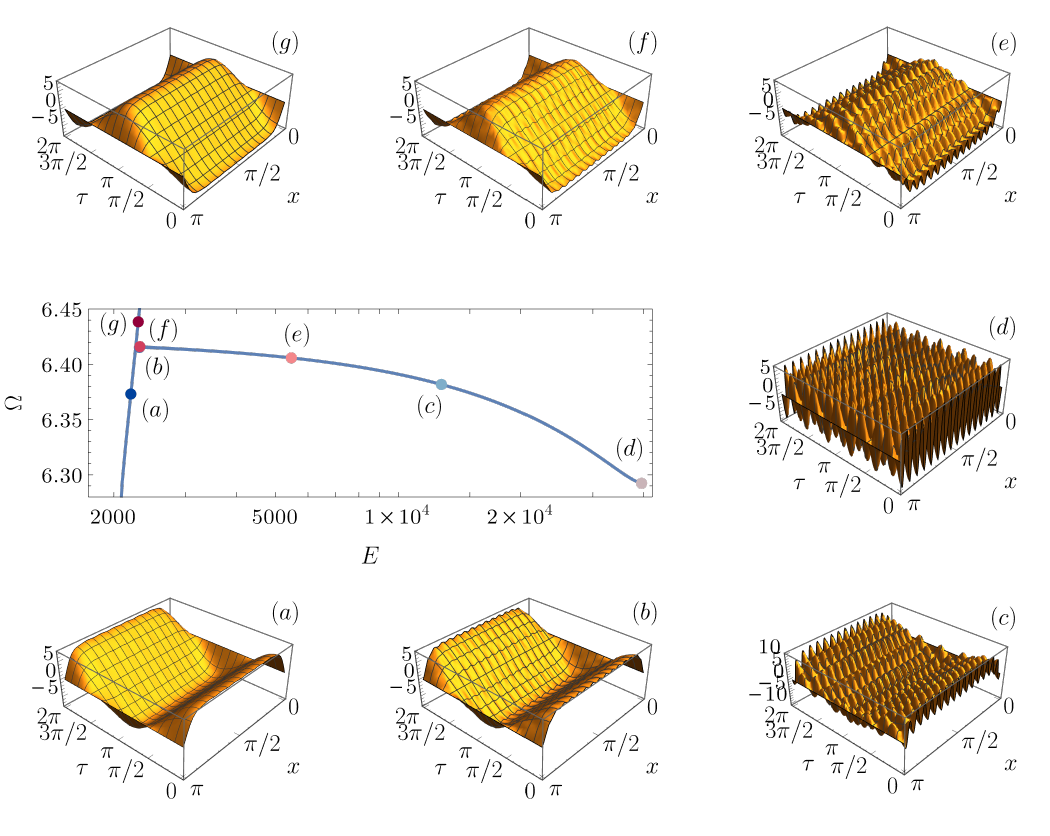}
	\caption{The spatio-temporal profiles of time-periodic solutions along one of the branches of Fig.~\ref{fig:EOmegaFull}. Following the solution curve (a)-(g), starting at the bottom left plot and going counterclockwise, we observe the qualitative change. Upon entering the branch (a)-(b) the high frequency mode $\cos{5\tau}\sin{31x}$ gets excited while the fundamental mode $\cos{\tau}\sin{x}$ is gradually suppressed. At the bifurcation point (d) the fundamental mode vanishes, and the former mode becomes dominant. Continuation beyond (d) sends the solution towards the main branch (g) through (e) and (f), this time with the fundamental mode having opposite sign. When leaving the branch, (f) and (g), the solution resembles, modulo sign, the profiles on (b) and (a) respectively. The branch consists of two very close lines intersecting at (d) on the $\Omega-E$ diagram, which are not distinguishable on the scale of the plot. Consequently, the points (b) and (f) appear to overlap in this plot.}
	\label{fig:3Dprofiles}
\end{figure}

The locations where the ``vertical branches'' (hereafter referred to simply as branches) appear to terminate and turn back towards the primary branch mark the bifurcation points. Using as the initial guess in the Newton's algorithm data which is a perturbation of the solution at a bifurcation point one can switch branches. Then, instead of being ``reflected'' from the bifurcation point and following the original solution path, one can jump onto a new branch (the \textit{secondary branch}), as shown in Fig.~\ref{fig:EOmegaBifurcation}. Tracing this new branch (red curve in Fig.~\ref{fig:EOmegaBifurcation}) backwards (decreasing energy) and also following it forward (increasing energy) we find that it resembles the rescaled, according to Eq.~\eqref{eq:24.05.28_01}, part of the path close to the bifurcation point $(\Omega,E)=(1,0)$ of the primary branch. Recall that the scaling symmetry, valid for the PDE \eqref{eq:23.02.15_02}, does not hold for the truncated system \eqref{eq:24.05.13_06}-\eqref{eq:24.05.13_07}.

A close inspection of solutions on branches reveals a drastic change in their profile as we follow the solution path. Specifically, as we enter a branch, higher harmonics associated with a particular branch, get excited while the fundamental mode gets suppressed. At the bifurcation point, the solution is dominated by one of the high modes and its higher harmonics. Moreover, moving away from the bifurcation point and exiting the branch, the situation reverses: the higher modes are suppressed, and the lowest mode gets amplified, albeit with the opposite sign as at the ``entrance point.'' This occurs in such a way that, at the point where the branch meets the primary branch, the solution resembles, modulo sign, the configuration before it landed on the branch. This behaviour is illustrated in Fig.~\ref{fig:3Dprofiles}, where we see a bifurcation to a solution based on the particular mode $\cos{5\tau}\sin{31x}$. Additionally, in Fig.~\ref{fig:Modes}, we plot the amplitudes of the fundamental mode and the high-frequency mode, which is dominant at the secondary branch. In more complicated branches, such as the two lowest plots in Fig.~\ref{fig:EOmegaFull}, the situation is analogous, but due to a more intricate structure, it is inconvenient for a clear presentation. On such a branch, several modes compete and get subsequently excited or suppressed as we move along the solution curve.

\begin{figure}[!t]
	\centering
	\includegraphics[width=0.475\textwidth]{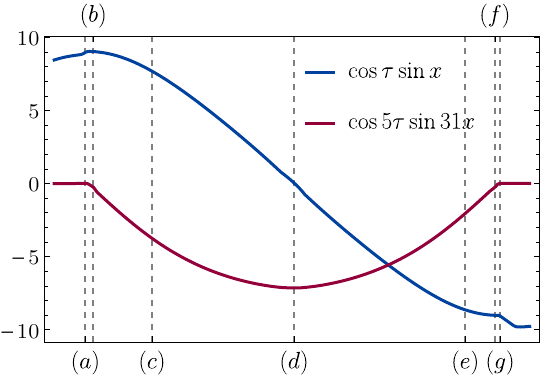}
	\caption{The amplitudes of the modes $\cos{\tau}\sin{x}$ and $\cos{5\tau}\sin{31x}$ along the solution curve, parametrised by a pseudo-arclength parameter, see discussion above Eq.~\eqref{eq:24.05.12_02}. The marked locations (a)-(g) correspond to the profiles shown in Fig.~\ref{fig:3Dprofiles}.}
	\label{fig:Modes}
\end{figure}

To conclude we comment on the reliability of the Galerkin approximation outlined in Sec.~\ref{sec:NumericalApproach} in the large amplitude regime. To demonstrate this, we compute the $L^{2}$ norm of the residual calculated by plugging the truncated Galerkin series \eqref{eq:24.05.13_02} into Eq.~\eqref{eq:23.02.15_02}. Next we check its behaviour with increasing truncation $N_{\tau}$, $N_{x}$. In general we observe exponential (spectral) convergence, however, its rate and certain details depend on the particular solution considered. We illustrate this in Fig.~\ref{fig:Residual} for two characteristic points of the $\Omega-E$ diagram: first located on the primary branch near the point where a new vertical branch emerges for $N_{\tau}=N_{x}=15$, second located at the end of the vertical branch (bifurcation point, c.f. Fig.~\ref{fig:EOmegaBifurcation}). Those tests, together with the comparison with perturbative expansion for small amplitudes, make us confident about the presented results.

\begin{figure}[!t]
	\centering
	\includegraphics[width=1.00\textwidth]{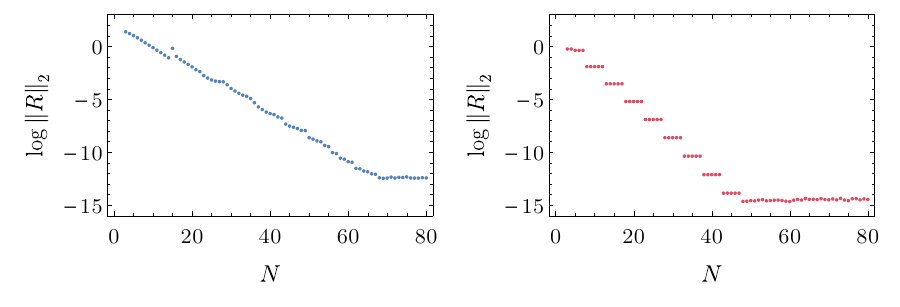}
	\caption{The $L^{2}$ norm of the residual $R$ of the truncated Galerkin series \eqref{eq:24.05.13_02} versus the truncation size $N=N_{\tau}=N_{x}$. Left: solution with $\Omega\approx 2.66911$ located on the primary branch. For $N=15$ we observe a jump of the residual due to the emergence of the vertical branch, which does not affect the overall spectral convergence. Right: solution at the end of the vertical branch (bifurcation point) with the dominant mode $\cos{3\tau}\sin{5x}$ and frequency $\Omega\approx 1.71514$. The stair-like structure can be explained by the spectral composition of this solution as it contains modes $\cos\left(3(2m+1)\tau\right)\sin\left(5(2n+1)\tau\right)$, $m\geq 0$, $n\geq 0$. In both cases we reach machine precision for large enough truncations.
    }
	\label{fig:Residual}
\end{figure}

\subsection{Focusing nonlinearity}
\label{sec:FocusingNonlinearityNum}
Finally, let us comment on the focusing equation. Using the transformation \eqref{eq:24.05.13_01}, adjusted to the phase shift adapted in the Galerkin system \eqref{eq:24.05.13_02}, we generate solutions from the already available data for the defocusing nonlinearity. As a result, the energy-frequency plot gets rescaled, and this scaling introduces significant qualitative changes, cf. Fig.~\ref{fig:EOmegaFull} and Fig.~\ref{fig:EOmegaFocusing}. In this case, the frequency $\Omega$ of small solutions decreases with amplitude (or equivalently with energy), which is in agreement with the analysis in Sec.~\ref{sec:FocusingNonlinearity}.

Along the primary branch the frequency drops to zero, and in the limit $\Omega\rightarrow 0$ we expect the solution to converge to the static profile given by Eq.~\eqref{eq:24.05.26_01}. Thus, the energy should approach $E_{1}$ in this limit. The apparently different behaviour of large solutions, as shown in Fig.~\ref{fig:EOmegaFocusing}, is the effect of a finite truncation of the Galerkin system. This is because the static solution $s_{1}$ is composed of infinitely many Fourier modes, and thus is not a solution of the truncated system. Similarly to the defocusing case, solutions bifurcating from higher eigenfrequencies can be obtained using the scaling \eqref{eq:24.05.28_01} with $m=1$. For $n>1$, the limiting solution as $\Omega\rightarrow 0$ will be $s_{n}$.

\begin{figure}[!t]
	\centering
	\includegraphics[width=0.98\textwidth]{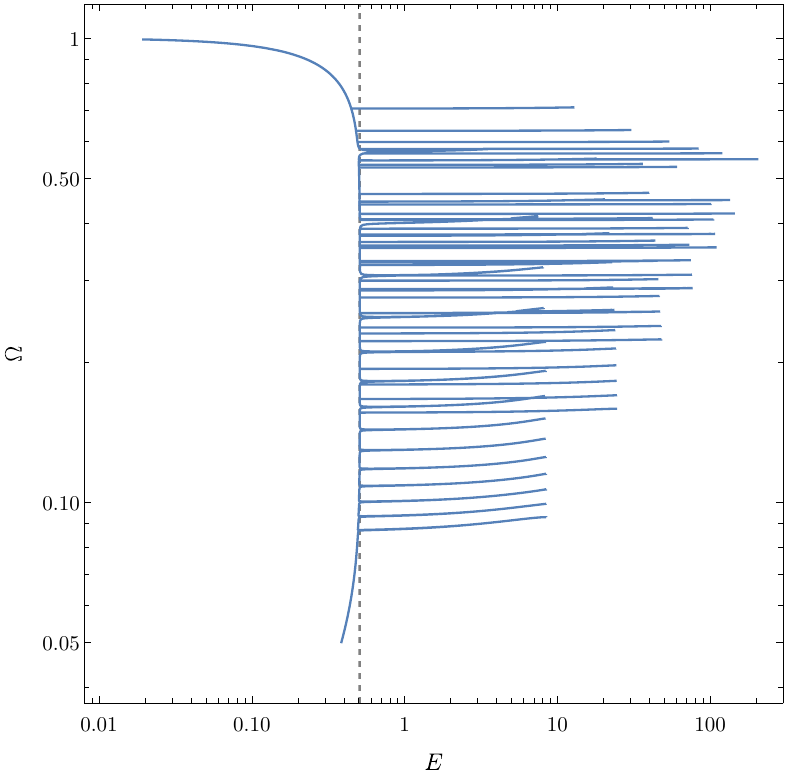}
	\caption{The energy-frequency diagram of time-periodic solutions for focusing nonlinearity (note the use of log-log scale). The plot was generated from the data shown in Fig.~\ref{fig:EOmegaFull} using the transformation \eqref{eq:24.05.13_01}. The dashed vertical line marks the energy of the static solution $E_{1}\approx 0.508157$, cf. \eqref{eq:24.05.26_01}. In the limit of $N_{x},N_{\tau}\rightarrow\infty$ we expect the curve to reach the point $(E_{1},0)$, corresponding to the static solution $s_{1}$, cf. Eq.~\eqref{eq:24.05.26_01}. However, for finite truncations the solution curve approaches the origin, as reaching the static solution requires infinitely many modes.}
	\label{fig:EOmegaFocusing}
\end{figure}

\section{Conclusions}
\label{sec:Conclusions}

The numerical results which we present were obtained by solving the Galerkin system truncated to 16 modes both in time and in space. Experiments with systems using larger truncations indicate that new and more intricate structures appear on the $\Omega-E$ diagram, however, features from smaller truncations remain largely unaffected by inclusion of more modes. For solutions of the PDE problem we expect the primary branch to have infinitely many branches, which on the $\Omega-E$ diagram will result in infinitely thin and complex structures. In particular, some branches will accumulate near $\Omega=1$ resulting in a Cantor-like set resembling the set of frequencies excluded in the rigorous results mentioned in the introduction. Let us point out that this structure is absent in Fig.~\ref{fig:SmallFrequency}, as for it to be present one would need to consider Galerkin systems of size greater than $\sim (10^8)^{2}$, see \cite{FicekB.2024}.

The new branches which emerge as we increase $N_{\tau}$ and $N_{x}$, get finer and finer and eventually become impossible to resolve numerically with a fixed $\Delta s$ set in \eqref{eq:24.05.12_02}. We expect that the diagram shown in Fig.~\ref{fig:EOmegaFull} already misses solutions, especially those located in the lower part of the diagram. Therefore, for practical reasons, we restricted our attention to relatively small truncations.

The equivalent diagram of Fig.~\ref{fig:EOmegaFull} for PDE, showing primary branch originating from $(\Omega,E)=(1,0)$ with associated branches, would present the minimal subset of time-periodic solutions. The complete diagram should include not only solutions bifurcating from higher eigenfrequencies, which could be generated using \eqref{eq:24.05.28_01} with $m=1$, but also secondary branches obtained by other rescalings \eqref{eq:24.05.28_01} of the primary branch. Those secondary branches will be attached to each of the (infinitely many) bifurcation points, cf.~Fig.~\ref{fig:EOmegaBifurcation}. Thus, the complete diagram of time-periodic solutions will have a fractal-like structure.

Our results suggest that solutions exist for all frequencies $\Omega>1$ ($\Omega<1$) for defocusing (focusing) nonlinearity. Solutions which are located on branches are drastically different from the ones on the primary branch. This might be directly related to the appearance of the restriction to the Cantor set of frequencies in rigorous results. In the forthcoming paper \cite{FicekB.2024}, we will provide a systematic analysis of the structure presented in Figs.~\ref{fig:EOmegaFull} and \ref{fig:EOmegaFocusing}, supporting our claims regarding the solutions of the PDE. As this analysis is beyond the scope of the current paper, it will be presented in a separate work.

\appendix

\vspace{8ex}

\section{Solution to the elliptic modulus equation}
\label{sec:DieckmannIdentity}
The goal of this Appendix is to find solutions to \eqref{eq:23.02.24_11}. Using \eqref{eq:24.06.03_01} and \eqref{eq:23.02.24_03} we can rewrite this relation in a more explicit way
\begin{align}\label{eq:04.06.24_1}
	\frac{\left(2k^2-1\right)K(k)}{2\pi}+3\sum_{n=1}^{\infty} \frac{q^{2n-1}}{(1+q^{2n-1})^{2}}=0
	\, ,
\end{align}
with $q$ given in \eqref{eqn:nome}.
Hence, we will prove the following

\begin{lemma}
Equation \eqref{eq:04.06.24_1} has on $[0,1]$ a unique solution given by $k=0.451075598810\ldots$.
\end{lemma}
\begin{proof}
The proof consists of two parts. First we show that one can transform \eqref{eq:04.06.24_1} to \eqref{eq:24.06.03_02}, an expression not involving the nome function $q$. Then we use inequalities for elliptic integrals to prove that \eqref{eq:24.06.03_02} has a unique solution on $[0,1]$.

We begin by stating the following identity
\begin{equation}
	\label{eq:23.02.24_12}
	\sum_{n=1}^{\infty} \frac{q^{2n-1}}{(1+q^{2n-1})^{2}} = \frac{K(k)}{2\pi^{2}}\left(E(k)-(1-k^{2})K(k)\right)
	\,.
\end{equation}
This relation can be found, for example, in the collection by Dieckmann \cite{web:dieckmann}, but we were unable to find its explicit proof in literature.
Hence, for the completeness, we give a short derivation based on elementary identities for a theta function \cite[\href{https://dlmf.nist.gov/20.4}{Sec.~20.2(i)}]{NIST:DLMF}
\begin{equation}\label{eqn:theta_function}
\theta_3 (z,q)=1+2\sum_{k=1}^\infty q^{n^2}\cos 2nz\, .
\end{equation}
Derivatives of $\theta_3$ with respect of the first and second variable will be denoted here with the use of $\partial_z$ and $\partial_q$, respectively. Then it holds
\begin{equation}\label{eqn:theta_bis}
\frac{\partial_z^2 \theta_3(0,q)}{\theta_3(0,q)}=-8\sum_{n=1}^{\infty} \frac{q^{2n-1}}{\left(1+q^{2n-1}\right)^2}\, ,
\end{equation}
see \cite[\href{https://dlmf.nist.gov/20.4}{Sec.~20.4}]{NIST:DLMF}, as well as
\begin{equation}\label{eqn:K_theta}
K(k)=\frac{\pi}{2}\theta_3^2(0,q)\, ,
\end{equation}
see \cite[\href{https://dlmf.nist.gov/20.9}{Sec.~20.9}]{NIST:DLMF}.

One can use a differential equation satisfied by $K$ \cite[\href{https://dlmf.nist.gov/19.4.i}{Sec.~19.4(i)}]{NIST:DLMF} to simplify the right hand side of \eqref{eq:23.02.24_12} to
\begin{equation}\label{eqn:App_rhs_1}
\frac{K(k)}{2\pi^{2}}\left(E(k)-(1-k^{2})K(k)\right)=\frac{k(1-k^2)}{2\pi^2}K(k) \frac{\diff{}}{\diff{k}}K(k)\,.
\end{equation}
While $K(k)$ can be simply expressed with $\theta_3$ using \eqref{eqn:K_theta}, for its derivative it holds
\begin{equation}\label{eqn:K_prime_theta}
 \frac{\diff{}}{\diff{k}} K(k)=\pi\, \theta_3(0,q)\, \partial_q \theta_3(0,q) \frac{\diff{q}}{\diff{k}}\, .
\end{equation}
To get the derivative of $q$ over $k$ we simply differentiate its definition \eqref{eqn:nome}
\begin{align*}
 \frac{\diff{q}}{\diff{k}}&=\pi q \left[-\frac{1}{K(k)} \frac{\diff{}}{\diff{k}} K\left(\sqrt{1-k^2}\right)+\frac{K\left(\sqrt{1-k^2}\right)}{K(k)^2} \frac{\diff{}}{\diff{k}}K(k) \right]\\
 &=\pi q \left[-\frac{1}{K(k)} \frac{k^2 K\left(\sqrt{1-k^2}\right)-E\left(\sqrt{1-k^2}\right)}{k\left(1-k^2\right)} +\frac{K\left(\sqrt{1-k^2}\right)}{K(k)^2}\frac{E(k)-\left(1-k^{2}\right)K(k)}{k\left(1-k^2\right)} \right]\\
 &=\frac{\pi q}{k\left(1-k^2\right)} \frac{1}{K(k)^2} \left[E\left(\sqrt{1-k^2}\right)K(k) + K\left(\sqrt{1-k^2}\right)E(k)-K\left(\sqrt{1-k^2}\right)K(k)\right]\\
 &=\frac{\pi^2 q}{2k\left(1-k^2\right)K(k)^2}\, .
\end{align*}
At the second step we have again used the differential equation for $K$, while the last equality is a result of the Legendre's relation \cite[\href{https://dlmf.nist.gov/19.7.1}{(19.7.1)}]{NIST:DLMF}.
In the end, by putting together \eqref{eqn:K_theta}, \eqref{eqn:App_rhs_1}, and \eqref{eqn:K_prime_theta}, the right hand side of \eqref{eq:23.02.24_12} becomes
\begin{equation}\label{eqn:rhs_2}
\frac{K(k)}{2\pi^{2}}\left(E(k)-(1-k^{2})K(k)\right)=\frac{q}{2}\frac{\partial_q \theta_3(0,q)}{\theta_3(0,q)}\, .
\end{equation}
At the same time, using \eqref{eqn:theta_bis} we can write the left hand side of \eqref{eq:23.02.24_12} simply as
\begin{equation}\label{eqn:lhs_1}
\sum_{n=1}^{\infty} \frac{q^{2n-1}}{(1+q^{2n-1})^{2}}= -\frac{1}{8} \frac{\partial_z^2 \theta_3(0,q)}{\theta_3(0,q)} \, .
\end{equation}
right hand sides of equations \eqref{eqn:rhs_2} and \eqref{eqn:lhs_1} are equal, as can be clearly seen from differentiation of the definition \eqref{eqn:theta_function}. This proves the identity \eqref{eq:23.02.24_12} and, as a result, lets us rewrite \eqref{eq:04.06.24_1} as \eqref{eq:24.06.03_02}, i.e.,
\begin{equation*}
	K(k)\left(6E(k)+(8k^{2}-7)K(k)\right) = 0
	\,.
\end{equation*}

Since we are interested in $k\in[0,1]$, where $K(k)\geq \pi/2>0$, to investigate roots of this equation it is enough to study the function $g$ defined as
\begin{equation}
	\label{eq:23.02.24_14}
	g(k) = 6E(k)+(8k^{2}-7)K(k)
	\,.
\end{equation}
Its values range from $g(0)=-\pi/2$ to $+\infty$ as $k\rightarrow 1$. Differentiation of $g$ leads, after application of differential equations satisfied for $E$ and $K$ \cite[\href{https://dlmf.nist.gov/19.4.i}{Sec.~19.4(i)}]{NIST:DLMF}, to
\begin{equation}
	\label{eq:23.02.24_15}
	g'(k) = \frac{(2k^2-1)E(k)+(1-k^2)(1+8k^2)K(k)}{k(1-k^2)}
	\,.
\end{equation}
Let us define $h(k)$ as the numerator of this expression. Elliptic integrals satisfy for $k\in (0,1)$ the following inequalities \cite[\href{https://dlmf.nist.gov/19.9}{Sec.~19.9}]{NIST:DLMF}
\begin{equation}
	\label{eq:23.02.24_16}
	\sqrt{1-k^2}\,K(k)<E(k)<\left(\frac{1+\sqrt{1-k^2}}{2}\right)^2 K(k)
	\,.
\end{equation}
Hence, for $k<\frac{1}{\sqrt{2}}$ it holds
\begin{equation}
	\label{eq:23.02.24_17}
	h(k)>\left(\frac{1}{4}(2k^2-1)(1+\sqrt{1-k^2})+(1-k^2)(1+8k^2)\right)K(k)>0
	\,,
\end{equation}
while for $k\geq\frac{1}{\sqrt{2}}$ one has
\begin{equation}
	\label{eq:23.02.24_18}
	h(k)\geq\left(\sqrt{1-k^2}+(1-k^2)(1+8k^2)\right)K(k)>0
	\,.
\end{equation}
It means that $g(k)$ is increasing in the interval $(0,1)$ and as a result it has exactly one zero. One can find numerically that it is located at $k=0.451075598810\ldots$, cf.~\cite{Vernov.1998,Khrustalev.2000,Khrustalev.2001}.
\end{proof}

\section{Explicit expression for source term at third order}
\label{sec:DerivationOfExplicitExpressionsForDjk}

In this Appendix we derive a compact formula for $B_{jk}$, i.e.\ non-diagonal Fourier coefficients of $\left(u^{(1)}\right)^3$, the source term in Eq.~\eqref{eq:23.02.15_05b}. Hence, we are looking for the following decomposition
\begin{equation}\label{eq:23.05.17_01}
\sum_{n,j,k=1}^\infty f_n f_j f_k \sin{nx}\sin{jx}\sin{kx}\sin{n\tau}\sin{j\tau}\sin{k \tau} = \sum_{N,M=1}^\infty B_{NM} \sin{Nx}\sin{M\tau}
\end{equation}
for $N\neq M$. Let us fix some positive $N$ and $M$ satisfying this condition. Using \eqref{eq:23.02.15_10} the left hand side of Eq.~\eqref{eq:23.05.17_01} can be expressed as a sum of sixteen terms of the type
\begin{equation*}
\sum_{n,j,k=1}^\infty f_n f_j f_k \, \sin(\pm n\pm j\pm k)x \, \sin(\pm n\pm j\pm k)\tau\, .
\end{equation*}
Each of these terms can be checked for the possibility of contributing to $\sin{Nx}\sin{M\tau}$, as we show below.

Let us begin by considering the term with $\sin(n+j-k)x \, \sin(n+j-k)\tau$. For it to contribute, we would need to have $n+j-k$ equal to $N$ or $-N$ and at the same time equal to $M$ or $-M$. It leads to either $M=N$ or $M=-N$ -- both of these alternatives give us a contradiction. Hence, this term does not contribute to the non-diagonal elements. The same reasoning can be repeated for the remaining three diagonal terms.

Now, let us investigate $\sin(n+j+k)x \, \sin(-n+j+k)\tau$. Assume that $n+j+k=N$ and $-n+j+k=M$. It is possible only for $N$ and $M$ of the same parity and $N>M$. Then we get $k=(N+M)/2-j$ and $n=(N-M)/2$ and the corresponding sum can be written as
\begin{equation*}
\sum_{j=1}^{\frac{N+M}{2}-1} f_{\frac{N-M}{2}} f_j f_{\frac{N+M}{2}-j} \, \sin N x \, \sin M \tau\, .
\end{equation*}
Alternatively, we can assume $n+j+k=N$ and $-n+j+k=-M$ leading to $n=(N+M)/2$ and $k=(N-M)/2-j$ and eventually giving the sum
\begin{equation*}
\sum_{j=1}^{\frac{N-M}{2}-1} f_{\frac{N+M}{2}} f_j f_{\frac{N-M}{2}-j} \, \sin N x \, \sin M \tau\, .
\end{equation*}
These two cases exhaust all possibilities here since $n+j+k=-N$ would require negative numbers.

Similar analysis repeated for the remaining terms in the end gives us
\begin{multline}
\label{eq:23.05.17_02a}
B_{NM}=\frac{3}{16}\left(-\sum_{j=1}^{\frac{N+M}{2}-1}f_{\frac{|N-M|}{2}} f_j f_{\frac{N+M}{2}-j}
+ \sum_{j=1}^{\frac{|N-M|}{2}-1}f_{\frac{N+M}{2}} f_j f_{\frac{|N-M|}{2}-j} 
\right.
\\
\left.
+ 2 \sum_{j=1}^{\infty}f_{\frac{|N-M|}{2}+j} f_j f_{\frac{N+M}{2}}
- 2 \sum_{j=\frac{N+M}{2}+1}^{\infty}f_{j-\frac{N+M}{2}} f_j f_{\frac{|N-M|}{2}}\right).
\end{multline}
This expression can be further simplified using the fact that $f_n$ is given by Eq.\ (\ref{eq:23.02.24_03}). As an example, let us consider 
\begin{equation*}
 \sum_{j=1}^{\infty}f_{\frac{|N-M|}{2}+j} f_j f_{\frac{N+M}{2}}= f_{\frac{N+M}{2}} \sum_{j=1}^{\infty}f_{\frac{|N-M|}{2}+j} f_j\, .
\end{equation*}
It is nonzero only if both $j$ and $|N-M|/2+j$ are odd, so it must hold $|N-M|\equiv 0$ mod $4$. For simplification let us introduce $\alpha=|N-M|/2$. Then we have
\begin{equation*}
 \sum_{j=1}^{\infty}f_{\alpha+j} f_j= \sum_{j=1}^\infty \frac{q^{(2j-1)+\alpha/2}}{\left(1+q^{2j-1}\right)\left(1+q^{\alpha+2j-1}\right)}\, .
\end{equation*}
Since
\begin{equation*}
\frac{q^{n}}{\left(1+q^{n}\right)\left(1+q^{\alpha+n}\right)}=\frac{1}{1-q^{\alpha}}\left(\frac{q^n}{1+q^n}-\frac{q^{n+\alpha}}{1+q^{n+\alpha}} \right)\, ,
\end{equation*}
this sum can be written as
\begin{equation*}
 \sum_{j=1}^{\infty}f_{\alpha+j} f_j= \frac{q^{\alpha/2}}{1-q^\alpha} \left(\sum_{j=1}^\infty \frac{q^{2j-1}}{1+q^{2j-1}} - \sum_{j=1}^\infty \frac{q^{{2j-1}+\alpha}}{1+q^{{2j-1}+\alpha} } \right)= \frac{q^{\alpha/2}}{1-q^\alpha}\sum_{j=1}^{|N-M|/4} \frac{q^{2j-1}}{1+q^{2j-1}}\,.
 \end{equation*}
 At the same time,
\begin{equation*}
\sum_{j=1}^{\frac{|N-M|}{2}-1}f_{\frac{N+M}{2}} f_j f_{\frac{|N-M|}{2}-j}=f_{\frac{N+M}{2}} \sum_{j=1}^{\frac{|N-M|}{2}-1}f_j f_{\frac{|N-M|}{2}-j}
\end{equation*}
is nonzero only if $|N-M|\equiv 0$ mod $4$. This sum can be simplified using
\begin{equation*}
\frac{1}{\left(1+q^{n}\right)\left(1+q^{\alpha-n}\right)}=\frac{1}{1-q^{\alpha}}\left(\frac{1}{1+q^{\alpha-n}}-\frac{q^n}{1+q^n} \right),
\end{equation*}
so
\begin{equation*}
\sum_{j=1}^{\frac{|N-M|}{2}-1}f_j f_{\alpha-j}=\frac{q^{\alpha/2}}{1-q^{\alpha}}\sum_{j=1}^{\frac{|N-M|}{4}}\left(\frac{1}{1+q^{\alpha-(2j-1)}}-\frac{q^{2j-1}}{1+q^{2j-1}} \right)
\,.
\end{equation*}
As a result, the second and third sums in Eq.\ \eqref{eq:23.05.17_02a} can be expressed as
\begin{multline}
\sum_{j=1}^{\frac{|N-M|}{2}-1}f_{\frac{N+M}{2}} f_j f_{\frac{|N-M|}{2}-j} 
+ 2 \sum_{j=1}^{\infty}f_{\frac{|N-M|}{2}+j} f_j f_{\frac{N+M}{2}}=
\\
\frac{q^{\beta/2}}{1+q^\beta}\frac{q^{\alpha/2}}{1-q^\alpha}\sum_{j=1}^{\alpha/2}\left(\frac{1}{1+q^{\alpha-(2j-1)}}+\frac{q^{2j-1}}{1+q^{2j-1}}\right)=\frac{\alpha}{2}\frac{q^{\beta/2}}{1+q^\beta}\frac{q^{\alpha/2}}{1-q^\alpha}\, ,
\end{multline}
where we have introduced $\beta=(N+M)/2$.

Analogously, two remaining sums are nonzero only for  $|N-M|\equiv 2$ mod $4$. They can be simplified in a similar way leading us to the following compact expressions
\begin{align}
	\label{eq:23.05.17_02b}
	B_{NM}=\begin{cases}
	\displaystyle\frac{3}{64} |N-M| \frac{q^{|N-M|/4}}{1-q^{|N-M|/2}}\frac{q^{(N+M)/4}}{1+q^{(N+M)/2}}\, ,\qquad & \mbox{if $|N-M|\equiv 0$ mod $4$}\,,\vspace{2ex}
	\\
	\displaystyle -\frac{3}{64} (N+M) \frac{q^{(N+M)/4}}{1-q^{(N+M)/2}}\frac{q^{|N-M|/4}}{1+q^{|N-M|/2}}\, , \qquad &\mbox{if $|N-M|\equiv 2$ mod $4$}\,.
	\end{cases}
\end{align}
These expressions can be further simplified using hyperbolic functions and ultimately giving Eq.\ (\ref{eq:23.03.14_001}).

\section{Resonance conditions at fifth order}
\label{sec:ResonanceConditionsAtThirdOrder}
As mentioned in Section \ref{sec:higher_orders}, to remove the resonances at fifth order it is sufficient to set $b_{jj}=0$ and $\xi_4=-\xi^2_2/2$. Here we would like to back up this observation. In particular, we show that $b_{jk}$ being antisymmetric indeed leads to the lack of resonances in the last term of the right hand side of (\ref{eq:23.02.15_05c}).

Let us expand the considered term
\begin{align}\label{eqn:AppB_Sum}
u^{(3)}\left(u^{(1)}\right)^2= \sum_{j,k,m,n=1}^{\infty}b_{jk}a_m a_n\cos j \tau \, \cos m \tau \, \cos n \tau\, \sin kx \, \sin mx \, \sin nx\, .
\end{align}
Then we can use \eqref{eq:3coss} and \eqref{eq:23.02.15_10} to write the expression under the sum as (after neglecting the factor $1/16$)
\begin{align*}
b_{jk}a_m a_n \left[ \cos(j + m + n) \tau+ \cos(j - m - n) \tau+ \cos(- j + m - n) \tau+ \cos(- j - m + n) \tau\right]\\
\times\left[ \sin (k + m + n)x+ \sin (k - m - n)x+ \sin (- k +m -n)x+ \sin (-k -m +n)x\right].
\end{align*}
Since we are interested in resonant terms, we need to focus on elements of this sum where the mode number inside one of the $\cos$ functions agrees with the mode number of one of the $\sin$ functions. For example, assume that for some fixed $j$, $k$, $m$, and $n$ it holds $j+m+n=-k+m-n$. Then $b_{jk} a_m a_n \cos(j + m + n) \tau\, \sin (-k + m - n)x$ is a resonant term. However, the sum (\ref{eqn:AppB_Sum}) also includes the term $b_{kj} a_m a_n \cos(-k + m - n) \tau\, \sin (j + m + n)x$. Since $b_{kj}=-b_{jk}$ both these terms cancel each other. The same argumentation holds for any other pair of $\cos$ and $\sin$ functions, also including potential differences in signs of their arguments. Hence, all resonant terms in \eqref{eqn:AppB_Sum} vanish as expected.

\vspace{4ex}

\printbibliography

\end{document}